\documentclass{amsart}

\usepackage{tikz}
\usepackage{amsfonts}
\usepackage{graphicx}
\usepackage{amssymb}
\usepackage{amsmath}
\usepackage{amsthm}
\usepackage{color}
\usepackage{hyperref,url}
\usepackage{mathrsfs}

\usepackage{tikz-cd}
\usepackage{newunicodechar}
\usepackage{graphics}
\usepackage{amsfonts}
\usepackage{color}
\usepackage{listings}
\usepackage{pstricks,pst-node,pst-tree,pstricks-add}

\definecolor{plum}{rgb}{.5,0,1}

\theoremstyle{plain}
\newtheorem{theorem}{Theorem}

\newtheorem{lemma}{Lemma}
\newtheorem{prop}{Proposition}
\newtheorem{question}{Question}

\theoremstyle{definition}
\newtheorem{definition}{Definition}

\theoremstyle{remark}
\newtheorem*{remark}{Remark}

\newcommand{\om}{\omega}
\newcommand{\dist}{\textup{dist}}

\newcommand{\Z}{\mathbb{Z}}

\newcommand{\R}{\mathbb{R}}

\newcommand{\e}{\epsilon}

\newcommand{\sgn}{\mathrm{sgn}}

\newcommand{\wind}{\mathrm{wind}}
\newcommand{\im}{\mathrm{Im}}
\newcommand{\Lip}{\mathrm{Lip}}
\newcommand{\de}{\delta}

\newcommand{\wt}{\widetilde}

\title{On Kakeya maps with regularity assumptions}
\author{Yuqiu Fu}
\email{yuqiufu@mit.edu}
\address{Deparment of Mathematics, MIT, Cambridge, MA 02139}

\author{Shengwen Gan}
\email{shengwen@mit.edu}
\address{Deparment of Mathematics, MIT, Cambridge, MA 02139}


\begin{document}

\maketitle
\begin{abstract}
In $\mathbb R^n$, we parametrize Kakeya sets using Kakeya maps. A Kakeya map is defined to be a map 
$$\phi:B^{n-1}(0,1)\times [0,1]\rightarrow \R^{n},\qquad (v,t)\mapsto (c(v)+tv,t),$$
where $ c:B^{n-1}(0,1)\rightarrow \R^{n-1}$. 
The associated {Kakeya set} is defined to be $ K:=\im (\phi). $

We show that the Kakeya set $K$ has positive measure if either one of the following conditions is true.

\noindent
(1) $c$ is continuous and $c|_{S^{n-2}}\in C^\alpha(S^{n-2})$ for some $\alpha>\frac{(n-2)n}{(n-1)^2}$,

\noindent
(2) $c$ is continuous and $c|_{S^{n-2}}\in W^{1,p}(S^{n-2})$ for some $p>n-2$. 
    
\end{abstract}

\section{Introduction}
     The Kakeya set conjecture says
if a set $ K\subset \R^n $ contains a unit line segment in each direction (such a set is called a {Kakeya set}), then $ K $ has Hausdorff dimension $ n $.
% For a detailed introduction to this conjecture, see \cite{wolff2003lectures}. 
The $n=2$ case is solved by Davies \cite{davies1971some}, so we shall restrict ourselves to $n\geq 3.$
Although we cannot solve the full conjecture, we can prove some positive results by assuming some regularity on the Kakeya set.

We start by defining the Kakeya map.
% which records the positions of every line segments in a Kakeya set. 
Notation-wise all the balls are closed balls. For example, by $B^{n-1}(0,1)$ we mean the closed unit ball in $\R^{n-1}$.

% \begin{definition}[Kakeya map]
%     For any function $ c:S^{n-1}\rightarrow \R^n $, we let $ \phi_c$ be the map 
%     $$\phi_c:S^{n-1}\times [0,1]\rightarrow \R^{n},\qquad (v,t)\mapsto c(v)+tv.$$
%     We call $ \phi_c $ a Kakeya map, and $K_c:=\im(\phi_c)=\bigcup_{v\in S^{n-1}}c(v)+[0,1]\cdot v$ the associated {{Kakeya set}}. 
% \end{definition}
%     Intuitively, we can think of $c$ as a ``direction-to-position" map, since for each direction $v$, $ K_c $ contains a $v$-direction unit line segment starting from position $c(v)$.
%     The Kakeya set conjecture can be then rephrased as for any Kakeya map $\phi$, its image $\im(\phi)$ has full Hausdorff dimension.

\begin{definition} [Kakeya map] \label{kakeyamaplocal}
    Given a direction map $ c:B^{n-1}(0,1)\rightarrow \R^{n-1} $, we define the associated {{Kakeya map}} to be the map 
   \begin{equation}\label{defphi}
       \phi:B^{n-1}(0,1)\times [0,1]\rightarrow \R^{n},\qquad (v,t)\mapsto (c(v)+tv,t).
   \end{equation}
    We define the associated {Kakeya set} to be $ K:=\im (\phi). $
\end{definition}

\begin{remark}
By construction, for any $v\in B^{n-1}(0,1)$, $K$ contains a line segment in direction $(v,1)$. Actually, this line segment in direction $(v,1)$ has one of its endpoint at $(c(v),0)$. This is the reason that we call $c$ the direction map. Sometimes, it is good to write $\phi=\phi_c$ to highlight the dependence on $c$, but we just omit the subscript and write as $\phi$ since $c$ is always priorly fixed and there is no ambiguity.
\end{remark}

Let us talk about the regularity assumptions that we will impose on $\phi$. First of all, we assume $c$ is a continuous map. Second, we assume $c|_{S^{n-2}}$ (which is the restriction of $c$ to $S^{n-2}$) lies in some function spaces of high regularity, for example, $C^\alpha(S^{n-2}), W^{1,p}(S^{n-2})$.

We state our main results.
\begin{theorem}\label{holderthm}
  If $c:B^{n-1}(0,1) \rightarrow \R^{n-1}$ is continuous and $c|_{S^{n-2}}$ is $\alpha$-H\"{o}lder continuous for some $\alpha>\frac{(n-2)n}{(n-1)^2},$ then $\im(\phi)$ has positive Lebesgue measure.
\end{theorem}

\begin{theorem}\label{sobolevthm}
  If $c:B^{n-1}(0,1) \rightarrow \R^{n-1}$ is continuous and $c|_{S^{n-2}}$ lies in $W^{1,p}(S^{n-2})$ for some $p>n-2,$ then $\im(\phi)$ has positive Lebesgue measure.
\end{theorem}

\begin{remark}
In a previous version of Theorem \ref{sobolevthm}, the regularity assumption was $c\in H^s(B^{n-1}(0,1))$ for some $s>(n-1)/2$. However, the definition of fractional Sobolev space on bounded domains as well as on manifolds is quite tricky and is not the main purpose of this paper, so we switch to a less tricky space $W^{1,p}(S^{n-2})$, which is defined by pulling back to the Euclidean space. To define $W^{1,p}(S^{n-2})$, we first choose two charts $\{U_1,U_2\}$ to cover $S^{n-2}$. Let $\psi_i:U_i\rightarrow B^{n-2}(0,1)$ be diffeomorphisms. For $f$ being a function on $S^{n-2}$, we define the norm:
\begin{equation}
    \|f\|_{W^{1,p}(S^{n-2})}:=\sum_{i=1}^2\|f\circ\psi_i^{-1}\|_{W^{1,p}(B^{n-2}(0,1))}.
\end{equation}
It is not hard to check by the chain rule that for different choices of the charts, the norms defined as above are comparable.
\end{remark}

The proofs will largely rely on the winding number from topology. In the rest of this section, we briefly discuss the winding number and its properties.
% let us talk about some notions from algebraic topology.
% Here we use two notions from algebraic topology, namely {\em{degree}} and {\em{winding number}}. In our paper, we only use the most fundamental form of them. 
\subsection{Winding number}
We first set up some notation. Given a continuous function
$$ f:S^{n}\rightarrow S^{n}, $$
we use $\deg f$ to denote the degree of $f.$
There are many ways to define the degree of a function, which all turn out to be equivalent. In \cite{hatcher2002algebraic} Section 2.2, the degree of $f$ is defined to be the integer $d$ such that the induced homomorphism 
$$ f_*:H_n(S^n)\rightarrow H_n(S^n) $$
satisfies $f_*(\alpha)=d\alpha$
(noting that $H_n(S^n)\approx\Z$).

We also note that $S^n$ and $\R^{n+1}\setminus \{0\}$ are homotopically equivalent, so $ H_n(S^n)\approx H_n(\R^n\setminus \{0\})$ are isomorphic.
Therefore, we can define the degree of 
$$f:S^n\rightarrow \R^{n+1}\setminus\{0\}$$
using homology groups in the same way. In this case, we usually call it the winding number of $f$ at $0$, denoted by $\wind(f,0)$. We can replace $0$ by any other point $x$ and define $\wind(f,x)$ as well.

We remark that we can also define the winding number in the following way.
Suppose we are given a continuous function
$$ g: S^{n-1}\rightarrow \R^n, $$
then for $z\in \R^n\setminus \im(g)$ we define
$$ g_z: S^{n-1}\rightarrow S^{n-1}, \quad z\mapsto \frac{g(x)-z}{|g(x)-z|}.$$
The winding number of $g$ at $z$ can be defined equivalently as
$$ \wind(g,z):=\deg g_z.$$
Morally speaking, $\wind(g,z)$ is the number of times that the ``hypersurface" $g(S^{n-1})$ wraps around $z$.

\medskip

Next, we interpret the winding number of $f$ from an analytic point of view. More precisely, we have the following lemma.

\begin{figure}
    \centering
    
\begin{tikzpicture}[scale=1.6]
\node (A) at (0,0) {$H_{n+1}(B^{n+1},B^{n+1}\setminus\{x_i\})$};
\node (B) at (3,1) {$H_{n+1}(B_i^{n+1},B^{n+1}_i\setminus \{x_i\})$};
\node (C) at (6,1) {$H_{n+1}(\R^{n+1},\R^{n+1}\setminus\{0\})$};
\node (D) at (3,0) {$H_{n+1}(B^{n+1},B^{n+1}\setminus \wt f^{-1}(0))$};
\node (E) at (6,0) {$H_{n+1}(\R^{n+1},\R^{n+1}\setminus\{0\})$};
\node (F) at (3,-1) {$H_n(S^n)$};
\node (G) at (6,-1) {$H_n(S^n)$};

\path[->,font=\scriptsize]
(B) edge node[above]{$\approx$} (A)
(B) edge node[above]{$\wt f_*$} (C)
(B) edge node[right]{$k_i$} (D)
(C) edge node[right]{$\approx$} (E)
(D) edge node[above]{$\wt f_*$} (E)
(D) edge node[above]{$p_i$} (A)
(F) edge node[above]{$\approx$} (A)
(F) edge node[right]{$j$} (D)
(F) edge node[above]{$\wt f_*$} (G)
(G) edge node[right]{$\approx$} (E);
\end{tikzpicture}

    \caption{}
    \label{diagram}
\end{figure}

\begin{lemma}\label{lemwind}
Assume 
$ f:S^{n}\rightarrow \R^{n+1}\setminus\{0\} $
is smooth. Consider another smooth function 
$$ \wt f: B^{n+1}(0,1)\rightarrow \R^{n+1} $$
satisfying $\wt f|_{S^n}=f$. Suppose $0\in\R^{n+1}$ is a regular value of $\wt f$ in the sense that $D\wt f(x)$ is nonsingular for any $x\in \wt f^{-1}(0)$. Then we have
\begin{equation}\label{degfsgn}
    \wind(f,0)=\sum_{x\in\wt f^{-1}(0)}\sgn(x).
\end{equation} 
Here, $\sgn(x)=1$ if $\det( D\wt f(x))>0$ and $=-1$ if $\det( D\wt f(x))<0$.
\end{lemma} 

One noticeable thing according to this lemma is that the right hand side of \eqref{degfsgn} only depends on the value of $\wt f$ on $S^{n}$. This lemma is fundamental from the point of view of algebraic topology, but we still provide the proof. 

\begin{proof}
Write $\wt f^{-1}(0)=\{x_1,\dots,x_m\}$. This is a finite set because $0$ is a regular value of $\wt f$. Of course, it could also be an empty set. We can find a small number $\de>0$ so that $\{B^{n+1}(x_i,\de)\}_{i=1}^m$ are disjoint and each $\wt f|_{B^{n+1}(x_i,\de)}$ is a diffeomorphism onto some neighborhood of $0$. For simplicity, we denote $B^{n+1}(0,1)$ by $B^{n+1}$ and denote $B^{n+1}(x_i,\de)$ by $B_i^{n+1}$.

Note that $\wt f$ induces a commutative diagram as in Figure \ref{diagram}. This is essentially the same as the diagram in \cite{hatcher2002algebraic} page 136.
We explain what those arrows mean. The arrows with $\approx$ mean that the relative homology groups are isomorphic. $k_i$ and $p_i$ are induced by inclusions. 
The top two groups are isomorphic to $\Z$, and the top homomorphism
$$ H_n(B_i^{n+1},B_i^{n+1}\setminus\{x_i\})\xrightarrow{\wt f_*}H_n(\R^{n+1},\R^{n+1}\setminus\{0\}) $$
becomes the multiplication by an integer called the \textit{local degree} of $\wt f$ at $x_i$, written $\deg \wt f|_{x_i}$. Since $\wt f|_{B^{n+1}_i}$ is a diffeomorphism, we have $\deg \wt f|_{x_i}=\sgn(x_i)$.
For the definition of homomorphism $j$, we first consider the inclusion
$$ (B^{n+1},S^n)\hookrightarrow (B^{n+1},B^{n+1}\setminus\wt f^{-1}(0)), $$
which induces the homomorphism
$$ H_{n+1}(B^{n+1},S^n)\rightarrow H_{n+1}(B^{n+1},B^{n+1}\setminus\wt f^{-1}(0)). $$
Since $H_{n+1}(B^{n+1},S^n)\approx H_n(S^n)$, we denote the above map by
$$ j: H_n(S^n)\rightarrow H_{n+1}(B^{n+1},B^{n+1}\setminus \wt f^{-1}(0)). $$
In the bottom homomorphism 
$$ H_n(S^n)\xrightarrow{\wt f_*}H_n(S^n), $$
it is a multiplication by an integer and this integer is exactly $\deg \wt f|_{S^n}$ which also equals the winding number $\wind(f,0)$ .
Similar to \cite{hatcher2002algebraic} Proposition 2.30, we can show that
$$ \wind(f,0)=\deg \wt f|_{S^n}=\sum_{i=1}^m \deg \wt f|_{x_i}=\sum_{i=1}^m\sgn(x_i). $$
This finishes the proof.
\end{proof}

\medskip

The way we connect the winding number with the Kakeya problem is through the following lemmas.

\begin{lemma}\label{lem1}
Given a continuous map $ g:B^n(0,1)\rightarrow \R^n,$ let $g|_{S^{n-1}}$ be the restriction of $g$ to $S^{n-1}$. For any $ z\notin g(S^{n-1}) $, we have $\wind(g|_{S^{n-1}},z)\neq 0 $ implies $ z\in \im(g). $
\end{lemma}

Lemma \ref{lem1} is a direct corollary of Lemma \ref{lemwind}. 
The next lemma is known as the isoperimetric inequality.
\begin{lemma}\label{lem2}
Given a smooth map $ g:S^{n-1}\rightarrow \R^n,$ we have
\begin{equation}\label{isoperimetricinequality}
\left( \int_{\R^n\setminus \im(g)} |\wind(g,z)|^{\frac{n}{n-1}} dz \right)^{\frac{n-1}{n}}\lesssim A(g). 
\end{equation}
Here, $A(g)=\int_{S^{n-1}}|\det(\sqrt{Dg^*(\omega)Dg(\omega)})|d\omega$ is the area of the self-intersecting ``hypersurface" $\im(g)$. $Dg^*$ denotes the transpose of $Dg$. Locally, $|\det(\sqrt{Dg^*(\omega)Dg(\omega)})|d\omega$ is the volume form on $\im(g)$.
\end{lemma}

Lemma \ref{lem2} can be found in an equivalent form as equation (2.10) of \cite{osserman1978isoperimetric}. We note that $z\mapsto\wind (g,z)$ is constant on each connected component of $\R^n\setminus \im(g)$. So if we denote the volumes these components by $\{V_k\}_k$ and the values of $\wind(g,z)$ on these components by $\{n_k\}_k$, we see that \eqref{isoperimetricinequality} is equivalent to
\begin{equation}\label{isoperimetricinequality2}
 (\sum_k |n_k|^{\frac{n}{n-1}}V_k)^{\frac{n-1}{n}}\lesssim A,
\end{equation}
which is equation (2.10) of \cite{osserman1978isoperimetric}.

Now we briefly discuss the main idea of the paper.
Given a continuous Kakeya map $\phi:B^{n-1}(0,1)\times [0,1]\rightarrow \R^{n}$, we want to show the Kakeya set $K=\im(\phi)$ has positive Lebesgue measure. It will be helpful to consider two maps: 
\begin{itemize}
    \item $ \phi_t(v):=\phi(v,t) $, the restriction of $\phi$ to the $t$-slice,
    \item $ \gamma_t:=\phi_t|_{S^{n-2}} $, the restriction of $\phi_t$ to the boundary sphere.
\end{itemize}

For $v_0\notin \im (\gamma_t),$ we denote by $\wind_t(v_0)$ the winding number of $ \gamma_t $ at point $v_0$, that is, the degree of the map 
$$S^{n-2}\rightarrow S^{n-2}, \, v\mapsto \frac{\gamma_t(v)-v_0}{|\gamma_t(v)-v_0|}.$$ 
If we can find some $t_0$ such that $\wind_{t_0}(v_0)\neq 0$ for some $v_0\notin \im(\gamma_{t_0})$, then by continuity, we have $v\notin \im (\gamma_t)$ and $\wind_{t}(v)\neq 0$ for $t$ close enough to $t_0$ and $v$ close enough to $v_0$. Thus by Lemma \ref{lem1}, we know there is an open neighborhood of $(t_0,v_0)$ contained in $\im(\phi)$, and hence it has positive Lebesgue measure. 

From the above discussion we see our main obstacle is the case that $|\im(\phi)|=0$ (so $\wind_t (v)$ is defined for almost every $v$) and the winding number is $0$ where it is defined. We will show this cannot happen if we assume some regularity property on the Kakeya map.

Let us first consider an easy case: $\phi$ is Lipschitz. By the area formula, we have
\begin{equation}\label{areaformula}
    \int_{B^{n-1}\times [0,1]} |\det (D\phi(x))|dx=\int_{\im(\phi)}\#\{x:\phi(x)=y\}dy.
\end{equation}
Using parameter $x=(v,t)$ and recalling the definition of $\phi$ in \eqref{defphi}, one can calculate that
$$ D_{v,t}\phi(v,t)=\left(
\begin{array}{cc}
     D_v c(v)+t I_{n-1}& 0  \\
     v& 1
\end{array}
\right). $$
Therefore, $\det(D\phi)=\det(D_{v,t}\phi(v,t))$ is a monic polynomial of degree $n-1$ in variable $t$. Therefore, $\det(D\phi)$ is nonzero almost everywhere. This implies \eqref{areaformula} is nonzero, and hence $|\im(\phi)|>0$. In fact we shall see in Section \ref{relatedsec} that using the winding numbers we can show $|\im (\phi)|$ is bounded from below by some positive constant depending only on the Lipschitz constant of $\phi.$

\medskip

Let us come back to Theorem \ref{holderthm} and Theorem \ref{sobolevthm} where the regularity assumption of the Kakeya map is weaker than Lipschitz.
The strategy is still proof by contradiction. We assume $\im(\phi)$ has zero Lebesgue measure. Then we use smooth maps to approximate H\"{o}lder continuous or Sobolev-regular Kakeya maps. We will eventually derive a contradiction using an isoperimetric inequality and the following key estimate:
\begin{equation}\label{666}
    1\lesssim \int_{0}^{1}|\int_{\R^{n-1}} \wind_{t}(x) dx|dt.
\end{equation}
One intuition for this to be true is that the inner integral is a polynomial in $t$ with leading term $t^{n-1}.$ This is very similar to an observation by Katz and Rogers in \cite{katz2018polynomial}. They showed that if $c$ is a polynomial of degree $d,$ then $|\im(\phi)|$ is bounded from below by some constant $c(n,d)>0.$ The key observation there is, by the area formula \eqref{areaformula} and Bezout's theorem, $|\im(\phi)|\geq C_d \int |\det (D\phi)|dvdt,$ and the latter integral is always at least $c(n)$ for some constant $c(n)>0.$  The difference in our approach is that instead of considering the integral of $|\det (D\phi)|$ in $(t,v),$ we consider the integral of $\int\det (D\phi_t)dv$ for each fixed $t.$ This is the signed volume on the $t$-slice and can be related to the winding number.

% \begin{lemma}\label{lem2}
% Given a smooth map $ g:S^{n-1}\rightarrow \R^n,$ we have
% \begin{equation}\label{isoperimetricinequality}
% \left( \int_{\R^n\setminus \im(g)} |\wind(g,z)|^{\frac{n}{n-1}} dz \right)^{\frac{n-1}{n}}\lesssim A(g). 
% \end{equation}
% Here $A(g)=\int_{S^{n-1}}|\det(\sqrt{Dg^*(\omega)Dg(\omega)})|d\omega$ is the area of the self-intersecting ``hypersurface" $\im(g)$. Locally, $|\det(\sqrt{Dg^*(\omega)Dg(\omega)})|d\omega$ is the volume form on $\im(g)$.
% \end{lemma}

% Lemma \ref{lem2} can be found in an equivalent form as equation (2.10) of \cite{osserman1978isoperimetric}. We  note that $z\mapsto\wind (g,z)$ is constant on each connected component of $\R^n\setminus \im(g)$. So if we denote these components by $\{V_k\}_k$ and the values of $\wind(g,z)$ on these components by $\{n_k\}_k$, we see \eqref{isoperimetricinequality} is equivalent to
% \begin{equation}\label{isoperimetricinequality2}
%  (\sum_k |n_k|^{\frac{n}{n-1}}V_k)^{\frac{n-1}{n}}\lesssim A.
% \end{equation}
% which is equation (2.10) of \cite{osserman1978isoperimetric}.

This paper is structured as follows.
In Section \ref{HolderSec} we will prove Theorem \ref{holderthm}, and in Section \ref{sobolevsec} we will prove Theorem \ref{sobolevthm}. Section \ref{relatedsec} will be a discussion of two other problems related to the Kakeya maps with regularity assumptions.

\vspace{5mm}
\noindent
\textbf{Notation.}\quad  We use $A\lesssim B$ to denote that $A\leq CB$ for constant $C$ which depends only on the dimension $n.$ $A\sim B$ will mean $A\lesssim B$ and $B\lesssim A.$ We will use $A\lesssim_{q} B$ to denote $A\leq C_{q}B$ for some constant $C_q$ depending on $q$ (and $n$). The closed unit ball in $\R^k$ is denoted by $B^k(0,1)$, and the unit sphere in $\R^k$ is denoted by $S^{k-1}.$

\section{H\"{o}lder continuous Kakeya map} \label{HolderSec}
Let $n\geq 3.$ Suppose we have a Kakeya map
$$\phi:B^{n-1}(0,1)\times [0,1]\rightarrow \R^{n},\qquad (v,t)\mapsto (c(v)+tv,t)$$
where $c:B^{n-1}(0,1)\rightarrow \R^{n-1}$ is continuous and $C^\alpha$ on $S^{n-2}$. We will prove the following theorem.

\begin{theorem}
  If $\alpha>\frac{(n-2)n}{(n-1)^2},$ then $\im(\phi)$ has positive Lebesgue measure.
\end{theorem}

We start with some definitions. Denote the restriction of $c$ to $S^{n-2}$ by $\underline{c}.$ Let $\underline{c}_\e$ be the $\e$-mollification of $\underline{c}.$ To be precise, let $\rho:\R^{n-1}\rightarrow \R$ be a radial  compactly supported smooth bump function in $\R^{n-1}$ adapted to $B^{n-1}(0,1),$ and let $$\rho_\e(y):=d_\e\e^{2-n}\rho(y/\e),$$
where the normalization constant $d_\e$ is set to be \begin{equation}\label{d}
    d_\e=\e^{n-2}(\int_{S^{n-2}} \rho ((y_0-y)/\e)dy)^{-1} 
\end{equation}
for any $ y_0\in S^{n-2}$. Note that the right hand side of \eqref{d} is independent of $y_0\in S^{n-2}$ since $\rho$ is a radial function. Also, we have $d_\e\sim 1$.

We write $\underline{c}$ in components as $\underline{c}=(\underline{c}_1,\cdots,\underline{c}_{n-1})$. Finally we define $\underline{c}_{\e}$ to be $(\underline{c}_1*\rho_{\e},\cdots,\underline{c}_{n-1}*\rho_\e)$, where $\underline{c}_i*\rho_\e(x):=\int_{S^{n-2}}\underline{c}_i(y)\rho_\e(x-y)dy$ for $x\in S^{n-2}$.
So the convolution $\underline{c}_i*\rho_\e(x)$ averages the value of $\underline{c}_i $ over an $\e$-neighborhood of $x$ on $S^{n-2}$.

Define $\gamma_t$ to be the map
$$S^{n-2}\rightarrow \R^{n-1},\qquad v\mapsto \underline{c}(v)+tv,$$
and $\gamma_{t,\e}$ to be the map
$$S^{n-2}\rightarrow \R^{n-1},\qquad v\mapsto \underline{c}_\e(v)+tv.$$
Then from the H\"{o}lder continuity of $c$ we have $|\gamma_t(v)-\gamma_{t,\e}(v)|\lesssim_c \e^\alpha,$ which implies:
\begin{equation}\label{dist}
    \wind_{t,\e}(x)=\wind_t(x) \quad \text{if}\quad \dist(x,\im \gamma_t)\gtrsim_c \e^\alpha.
\end{equation}
Here $\wind_{t}(x)=\wind(\gamma_t,x)$ is the degree of the map
\begin{equation}\label{defwind1}
    S^{n-2}\rightarrow S^{n-2},\qquad y\rightarrow \frac{\gamma_t(y)-x}{|\gamma_t(y)-x|},
\end{equation} 
and $\wind_{t,\e}(x)=\wind(\gamma_{t,\e},x)$ is the degree of the map
\begin{equation}\label{defwind2}
    S^{n-2}\rightarrow S^{n-2},\qquad y\rightarrow \frac{\gamma_{t,\e}(y)-x}{|\gamma_{t,\e}(y)-x|}.
\end{equation} 

We will need the following identity from differential topology that relates the integral of winding numbers with the integral of the determinant of the differential.

\begin{lemma}\label{intwindlem}
Suppose $f:S^{n-2}\rightarrow \R^{n-1}$ is smooth.
Let $\wt f: B^{n-1}(0,1)\rightarrow \R^{n-1}$ be a smooth map satisfying $\wt f|_{S^{n-2}}=f$.  
  Then, we have
  \[\int_{\R^{n-1}} \wind(f,x) dx=\int_{B^{n-1}(0,1)} \det(D \wt f(y)) dy.\]
\end{lemma}
\begin{proof}
We recall that we say $x\in \R^{n-1}$ is a regular value of $f$ if for any $y\in f^{-1}(x)$ we have $\det(Df(y))\neq 0$. In particular, $x$ is a regular value if $f^{-1}(x)$ is an empty set.

By Lemma \ref{lemwind}, if $x\in \R^{n-1}\setminus \im(f)$ is a regular value of $\wt f$, then
  \[\wind(f,x)=\sum_{y\in \wt f^{-1}(x)} \sgn (y),\]
  where $\sgn(y)$ equals $1$ if $\det(D \wt f(y))>0,$ and $-1$ if $\det(D\wt f(y))<0$.
  We also note that by Sard's theorem almost every $x\in \R^{n-1}$ is a regular value of $\wt f$. So,
  \begin{align*}
      \int_{\R^{n-1}} \wind(f,x) dx & =\int_{\R^{n-1}} \sum_{y\in f^{-1}(x)} \sgn (y) dx \\
      & = \int_{B^{n-1}(0,1)} \det(D\wt f(y)) dy.
  \end{align*}
(Rigorously speaking, we should write the integration domain as $\R^{n-1}\setminus \im(f)$. But since $\im(f)$ has zero measure, we still write it as $\R^{n-1}$ without any ambiguity.)  
\end{proof}

\begin{remark}
From the previous lemma, we see that the value of the integral 
$$\int_{B^{n-1}(0,1)} \det(D \wt f(y)) dy$$
only depends on the value of $\wt f$ on $S^{n-2}$. Therefore, for any smooth function $f:S^{n-2}\rightarrow \R^{n-1}$, it makes sense to define the integral
$$\int_{B^{n-1}(0,1)} \det(D f(y)) dy.$$
\end{remark}

\begin{prop}\label{deglowerbnd}

Let $\wind_{\e,t}(x)$ be defined as in \eqref{defwind2}. We have
  \begin{equation}\label{deglowerbound}
    1\lesssim \int_{0}^{1}|\int_{\R^{n-1}} \wind_{t,\e}(x) dx|dt.
  \end{equation}
  The implicit constant is independent of $\e.$
\end{prop}
\begin{proof}
    Fix $\e>0.$ By Lemma \ref{intwindlem} and the remark above, we have
    $$\int_{\R^{n-1}} \wind_{t,\e}(x) dx =\int_{B^{n-1}(0,1)}\det(D\gamma_{t,\e}(v))dv=\int_{B^{n-1}(0,1)}\gamma_{t,\e}^*(dx_1\wedge\cdots dx_{n-1}).
    $$ 
    Here $\gamma_{t,\e}^*(dx_1\wedge\cdots dx_{n-1})$ is the pullback of the differential form.
    Further by 
    % the diffeomorphism invariance of the differential form and 
    Stokes' theorem, we have
   \begin{align*}
       \int_{B^{n-1}(0,1)}\gamma_{t,\e}^*(dx_1\wedge\cdots dx_{n-1})&=\int_{B^{n-1}(0,1)}d\big(\gamma_{t,\e}^*(x_1 dx_2\wedge\cdots dx_{n-1})\big)\\
       &=\int_{S^{n-2}} \gamma_{t,\e}^*(x_1 dx_2\wedge\cdots dx_{n-1}).
   \end{align*}
   
    If we write $\underline{c}_\e(v)=(\underline{c}_{1,\e}(v),\ldots,\underline{c}_{n-1,\e}(v)),$ we have 
    $$\gamma_{t,\e}(v)=(tv_1+\underline{c}_{1,\e}(v),\ldots,tv_{n-1}+\underline{c}_{n-1,\e}(v)).$$
    Therefore
    \begin{align*}
        \int_{S^{n-2}} \gamma_{t,\e}^*(x_1 dx_2\wedge\cdots dx_{n-1}) & = \int_{S^{n-2}} (tv_{1}+\underline{c}_{1,\e}(v)) d(tv_2+\underline{c}_{2,\e}(v)) \wedge \cdots \wedge d(tv_{n-1}+\underline{c}_{n-1,\e}(v)) \\
        & =t^{n-1} \int_{S^{n-2}} v_{1} dv_2 \wedge \cdots \wedge dv_{n-1} + e(t) \\
        & =t^{n-1} + e(t),
    \end{align*}
    where $e(t)$ is a polynomial in $t$ with degree at most $n-2$ and coefficients determined by $c,\e.$
    In summary, we have shown
    $$ \int_{\R^{n-1}} \wind_{t,\e}(x) dx= t^{n-1}+e(t). $$
    
    Next we claim that 
    \begin{equation}\label{claim}
        \int_{0}^1 \big|t^{n-1}+e(t)\big| dt\gtrsim 1.
    \end{equation}
    If the claim \eqref{claim} is true, we may then conclude
    $$1\lesssim \int_{0}^1 \big|t^{n-1}+e(t)\big| dt= \int_{0}^{1}\big|\int_{\R^{n-1}} \wind_{t,\e}(x) dx\big|dt.$$
    
    So it suffices to prove the claim \eqref{claim}. The proof can be found in \cite{katz2018polynomial}, but we give the proof here for the sake of completeness. Since $t^{n-1}+e(t)$ is a monic polynomial of degree $n-1$, we write
    $$ t^{n-1}+e(t)=(t-r_1)\cdots (t-r_{n-1}), $$
    where $r_1\cdots r_{n-1}$ are complex numbers. We observe that for $t$ in a subset of $[0,1]$ with measure greater than $1/4,$ we have the estimate
    $$ |t-r_j|\ge \frac{1}{4n}\gtrsim 1,\ \ \ j=1,\cdots n-1, $$
    which means for $t$ in a subset of $[0,1]$ with measure greater than $1/4$ we have
    $$ |t^{n-1}+e(t)|=|t-r_1|\cdots |t-r_{n-1}|\gtrsim 1. $$
    Therefore, \eqref{claim} holds.
\end{proof}

\begin{proof}[Proof of Theorem 3]
Let $\underline{\phi}$ be the restriction of $\phi$ to $S^{n-2}\times [0,1].$ 
For the sake of contradiction we assume that $|\im \underline{\phi}|=0$ and $\wind_t(x)=0$ where it is defined. Then by Proposition \ref{deglowerbnd} and \eqref{dist} we have
\begin{equation}\label{2}
  \int_{0}^{1}\int_{\mathcal{N}_{C\e^\alpha}(\im \gamma_t)} |\wind_{t,\e}(x)|dxdt\gtrsim 1.
\end{equation}
Here $\mathcal{N}_{C\e^\alpha}(\im \gamma_t)$ denotes the $C\e^\alpha$-neighborhood of $\im \gamma_t$. $C$ is a constant depending on the H\"{o}lder constant of $c.$

Next we show
\begin{equation}\label{measNgammat}
    |\mathcal{N}_{C\e^\alpha}(\im \gamma_t)|\lesssim \e^{-n+2}(\e^{\alpha(n-1)})=\e^{(n-1)\alpha-n+2},
\end{equation}
and
\begin{equation}\label{areagammate}
    A(\gamma_{t,\e})\lesssim \e^{(\alpha-1)(n-2)}.
\end{equation}
(For the definition of $A(\gamma_{t,\e})$, see the line next to equation (\ref{isoperimetricinequality}).)

Indeed to see \eqref{measNgammat}, we choose a maximal $\e$-separated subset $S$ of $S^{n-2},$ so $|S|\sim \e^{2-n}$. We claim that the union of balls $\bigcup_{x_i\in S}B_{C_1\e^{\alpha}}(\gamma_t(x_i))$ covers $\mathcal{N}_{C\e^\alpha}(\im \gamma_t)$, when $C_1$ is large enough. In fact for any $y\in \mathcal{N}_{C\e^\alpha}(\im \gamma_t)$, there exists an $x\in S^{n-2}$, such that $|y-\gamma_t(x)|\lesssim \e^{\alpha}$. Also by the choice of $S$, there exists an $x_i\in S$ such that $|x-x_i|\leq \e$. So by the H\"{o}lder continuity, we have 
$$|y-\gamma_t(x_i)|\leq |y-\gamma_t(x)|+|\gamma_t(x)-\gamma_t(x_i)|\leq C_1\e^{\alpha},$$ 
which means $y\in\mathcal{N}_{C_1\e^{\alpha}}(\gamma_t(x_i))\subset \bigcup_{x_i\in S}B_{C_1\e^{\alpha}}(\gamma_t(x_i))$ if $C_1$ is sufficiently large. So \eqref{measNgammat} follows.

To see \eqref{areagammate}, we only need to show $|\nabla\gamma_{t,\e}|\lesssim \e^{\alpha-1}$ (here $\nabla$ is the gradient on $S^{n-2}$), which will imply 
$$A(\gamma_{t,\e})\lesssim\int_{S^{n-2}}|\det (D\gamma_{t,\e})|\lesssim \int_{S^{n-2}}|\nabla \gamma_{t,\e}|^{n-2}\lesssim \e^{(\alpha-1)(n-2)}.$$ 
% \begin{remark}
% Note that $A(\gamma_{t,\e})$ is defined to be the $(n-2)$-dimensional area of $\gamma_{t,\e}(S^{n-2})$ where the overlap is counted with multiplicity. We don't need a precise fomula for $A(\gamma_{t,\e})$. A bound of derivatives of $\gamma_{t,\e}$ suffices to give an upper bound of $A(\gamma_{t,\e})$.
% \end{remark}
Recall that $\rho$ is the mollifier with $\rho_\e(y)=d_\e\e^{2-n}\rho(y/\e)$ and $d_{\e}\sim 1$, so we have
\begin{align*}
    |\nabla \gamma_{t,\e}(y)| & = d_\e|\nabla\int_{S^{n-2}}\gamma_{t,\e}(x)\e^{2-n}\rho(\frac{y-x}{\e})dx| \\
    & \sim|\int_{S^{n-2}}\gamma_{t,\e}(x)\e^{2-n}\nabla(\rho(\frac{y-x}{\e}))dx|\\
    & = 
    |\int_{S^{n-2}}(\gamma_{t,\e}(x)-\gamma_{t,\e}(y))\e^{1-n}\nabla\rho(\frac{y-x}{\e})dx|\\
    & \lesssim \int_{\mathcal{N}_{C\e}(x)\cap S^{n-2}}|\gamma_{t,\e}(x)-\gamma_{t,\e}(y)|\e^{1-n}dy \\
    & \lesssim \int_{\mathcal{N}_{C\e}(x)\cap S^{n-2}}|x-y|^{\alpha}\e^{1-n}dy \\
    & \lesssim \e^{\alpha-1}.
\end{align*}
So we indeed have \eqref{areagammate}. 

Lemma \ref{lem2} states
$$\left(\int_{\R^{n-1}} |\wind_{t,\e}(x)|^{\frac{n-1}{n-2}} dx \right)^{n-2} \lesssim A(\gamma_{t,\e})^{n-1}.$$
Combining what we have so far with H\"{o}lder's inequality we obtain
\begin{align*}
    1 & \lesssim \int_{0}^{1}\int_{\mathcal{N}_{C\e^\alpha}(\im(\gamma_t))} |\wind_{t,\e}(x)|dxdt \\
      & \leq
|\mathcal{N}_{C\e^\alpha}(\im(\gamma_t))|^{1/(n-1)} \left( \int_{0}^{1}\int_{\mathcal{N}_{C\e^\alpha}(\im(\gamma_t))} |\wind_{t,\e}(x)|^{(n-1)/(n-2)}dxdt \right)^{(n-2)/(n-1)} \\
      & \lesssim \e^{((n-1)\alpha-n+2)/(n-1)}\e^{(\alpha-1)(n-2)} \\
      & =\e^{\alpha-\frac{n-2}{n-1}+(\alpha-1)(n-2)}.
\end{align*}
This is a contradiction if $\alpha-\frac{n-2}{n-1}+(\alpha-1)(n-2)>0,$ that is, $\alpha>\frac{(n-2)n}{(n-1)^2}.$
\end{proof}

% Actually, we believe we can drop the regularity of the Kakeya map to $C^{\alpha}$ for any small $\alpha>0$, and still get the same result. We make it a conjecture.

% \begin{conjecture}
% For any $\alpha>0$, if the Kakeya map is $C^{\alpha}$, then the associated Kakeya set has positive measure.
% \end{conjecture}

\section{Sobolev regular Kakeya map}\label{sobolevsec}

We use the same notation as in Section \ref{HolderSec} but instead of assuming $c|_{S^{n-2}}$ is H\"{o}lder continuous $C^\alpha,$ we assume $c$ is continuous and $c|_{S^{n-2}}$ lies in the Sobolev space $W^{1,p}(S^{n-2})$ for some $p>n-2.$ We write $p=n-2+\delta$ for some small $\delta>0.$ 

To compare this regularity assumption with that in Theorem \ref{holderthm}, by the Sobolev embedding, we know that $c|_{S^{n-2}}$ is $\delta'$-H\"{o}lder continuous for some $\delta'=O(\de)$. On the other hand when $\delta$ is small the space $ W^{1,p}$ and  $C^\alpha$ for $\alpha>\frac{(n-2)n}{(n-1)^2}$ are mutually non-inclusive. 

We will prove the following theorem.

\begin{theorem}\label{sobolev}
    If $c$ is continuous and $c|_{S^{n-2}}\in W^{1,n-2+\delta}(S^{n-2})$ for some $\delta>0$ then $\im(\phi)$ has positive Lebesgue measure.
\end{theorem}

\begin{proof}
Since $c|_{S^{n-2}} \in W^{1,n-2+\delta}(S^{n-2})$, we also have $\gamma_t \in W^{1,n-2+\delta}(S^{n-2}).$ 
We consider the mollified $c_\e$ as we did in Section \ref{HolderSec}. Also recall the definitions of $\wind_t(x)$ and $\wind_{t,\e}(x)$ in \eqref{defwind1} and \eqref{defwind2}.

By Proposition \ref{deglowerbnd}, we have
\begin{equation}\label{convergence}
    \int_{0}^{1}|\int_{\R^{n-1}} \wind_{t,\e}(x) dx|dt 
    \gtrsim 1.
    %\rightarrow  %\int_{0}^{1}|\int_{\R^{n-1}} \wind_{t}(x) dx|dt \text{ as } \e \rightarrow 0
\end{equation}
Suppose for the sake of contradiction that $|\im(\phi)|=0$ and $\wind_t(x)=0$ wherever it is defined.
Since by the Sobolev embedding $c|_{S^{n-2}}$ is in some H\"older space $C^{\delta'}$ for some $\delta '>0$, the same reasoning as in \eqref{dist} yields
\begin{equation}\label{distance}
    \wind_{t,\e}(x)=\wind_{t}(x)\quad \text{if}\quad \dist(x,\im \gamma_t ) \gtrsim_c \e^{\delta'}.
\end{equation}
So if we split 
\begin{align*}
    \int_{0}^{1}|\int_{\R^{n-1}} \wind_{t,\e}(x) dx|dt & \leq\int_{0}^{1}|\int_{\mathcal{N}_{C\e^{\delta'}}(\im \gamma_t)} \wind_{t,\e}(x) dx|dt
+\int_{0}^{1}|\int_{\R^{n-1}\setminus\mathcal{N}_{C\e^{\delta'}}(\im \gamma_t)} \wind_{t,\e}(x) dx|dt \\
   & =:I_1+I_2,
\end{align*}
then 
$$I_2=\int_{0}^{1}|\int_{\R^{n-1}\setminus\mathcal{N}_{C\e^{\delta'}}(\im \gamma_t)} \wind_{t,\e}(x) dx|dt =0,$$
since we assumed $\wind_{t}(x)=0$ and we have \eqref{distance}.

In the rest of the proof we will show that $I_1\rightarrow 0$ as $\e \rightarrow 0$, which contradicts \eqref{convergence}. 
We claim that
\begin{equation}\label{claim2}
    A(\gamma_{t,\e}) :=\int_{S^{n-2}} |\det \left(\sqrt{(D_{\omega}(c*\rho_\e +t\om))^*(D_{\omega}(c*\rho_\e +t\om))} \right)|d\omega\lesssim_c 1 .
\end{equation}
Here the notation is from Lemma \ref{lem2}. In the integral, we think of $c*\rho_\e+t\om$ as a function $S^{n-2}\rightarrow \R^{n-1}$.

To prove \eqref{claim2}, we cover $S^{n-2}$ by two coordinate charts $\{U_1,U_2\}$ with coordinate maps $\psi_i: U_i\rightarrow B^{n-2}(0,1)$. Note that the function $t\om$ is smooth and $\|c*\rho_{\e}\|_{W^{1,p}(S^{n-2})}\lesssim \|c\|_{W^{1,p}(S^{n-2})}$. So after change of variables and pulling back using $\psi_i$, the inequality \eqref{claim2} becomes $\int_{B^{n-2}(0,1)} |\det (\sqrt{(Df)^*(Df)} )|\lesssim_c 1$, where $f$ is a function on $B^{n-2}(0,1)$ satisfying $ \|f\|_{W^{1,p}(B^{n-2}(0,1))}\lesssim 1+\|c\|_{W^{1,p}(S^{n-2})} $. Expanding the integrand we see that 
$$\int |\det (\sqrt{(Df)^*(Df)})|=\int |\det ({(Df)^*(Df)})|^{\frac{1}{2}}\lesssim \int|Df|^{n-2} \lesssim \|f\|_{W^{1,n-2}}^{n-2}\lesssim_c 1.$$
So, we prove the claim \eqref{claim2}.

Therefore applying the isoperimetic inequality (Lemma \ref{lem2}) gives us
\begin{multline*}
    |\int_{\mathcal{N}_{C\e^{\delta'}}(\im \gamma_t)} \wind_{t,\e}(x) dx|\lesssim |\mathcal{N}_{C\e^{\delta'}}(\im \gamma_t)|^{1/(n-1)} \|\wind_{t,\e}\|_{L^{(n-1)/(n-2)}} \\
    \lesssim |\mathcal{N}_{C\e^{\delta'}}(\im \gamma_t)|^{1/(n-1)} A(\gamma_{t,\e}) \lesssim_c |\mathcal{N}_{C\e^{\delta'}}(\im \gamma_t)|^{1/(n-1)}.
\end{multline*}
In the last inequality, we used \eqref{claim2}.

Since 
$$\int_{0}^{1}|\mathcal{N}_{C\e^{\delta'}}(\im \gamma_t)|^{1/(n-1)}dt \lesssim \left(\int_{0}^{1}|\mathcal{N}_{C\e^{\delta'}}(\im \gamma_t)|dt\right)^{1/(n-1)} \le |\mathcal{N}_{C\e^{\delta'}}(\im \underline{\phi})|^{1/(n-1)},$$
we conclude
$$ I_1\lesssim_c |\mathcal{N}_{C\e^{\delta'}}(\im \underline{\phi})|^{1/(n-1)} \rightarrow 0$$
as $\e\rightarrow 0$ (because by assumption $|\im \underline{\phi}|=0$). Hence we finish the proof of Theorem \ref{sobolev}.
\end{proof}

\section{Other results related to the Kakeya problem}\label{relatedsec}

In this section we discuss two Kakeya-type problems which are under different settings from the previous sections, and may be of independent interest.

\subsection{Tube-Kakeya set with Lipschitz spacing condition}

In the previous sections, we studied the Kakeya set which is the union of line segments. In this subsection, we study the tube-version of the Kakeya set which is the union of $\de$-tubes.

\begin{definition}[Tube-Kakeya set]\label{defliptube}
	For $0<\de<1$, we choose $ V_\de $ to be a maximal $ \delta $-separated subset of $ B^{n-1}(0,1) $. For a map $ c:B^{n-1}(0,1)\rightarrow \R^{n-1} $, we consider the set of tubes $ \{ T_{c(v),v}^\de \}_{v\in V_\de} $, where $  T_{x,v}^\de $ is the $ \delta$-neighborhood of the segment $\{ (x+tv,t): t\in [0,1]\} $.
	We call the union of these tubes
	$$ \bigcup_{v\in V_\de} T^\de_{c(v),v} $$
	the tube-Kakeya set.
\end{definition}

It is conjectured that for any map $c$ we have
$$ |\bigcup_{v\in V_\de} T^\de_{c(v),v}|\gtrsim_\e \de^\e, $$
for $\e>0$.
In this subsection, we will assume some regularity on the map $c$ and prove the result.
We define the Lipschitz constant of $c$ by $$\|c\|_{\Lip}:=\max_{v,v'\in B^{n-1}(0,1)}\frac{|c(v)-c(v')|}{|v-v'|}.$$
Our result is:
	
\begin{prop}\label{propliptube}
	If $ \{ T^\de_{c(v),v} \}_{v\in V_\de} $ is a collection of tubes as in Definition \ref{defliptube}, then $$ |\bigcup_{v\in V_\de}T^{\de}_{c(v),v}|\gtrsim (\|c\|_{\Lip}+1)^{-(n-1)^2}.  $$ 
	%In particular, if $ \{ T^c_{v,\delta} \}_{v\in S} $ satisfies the Lipschitz spacing condition, then $ |\bigcup_{v\in S}T^c_{v,\delta}|\gtrsim 1$.
\end{prop}
\begin{proof}
    Consider the map 
	$$ \phi: B^{n-1}(0,1)\times [0,1]\rightarrow \R^n,\qquad (v,t)\mapsto (c(v)+tv,t), $$
	which is the Kakeya map corresponding to $c$. 
	
	Denote the Lipschitz constant $L=\|c\|_{\Lip}$.
	Note that $ B^{n-1}(0,1)\subset \bigcup_{v\in V_\de}B(v,2\delta), $ and $ \phi(B(v,2\delta)\times [0,1])\subset T^{100L\delta}_{c(v),v} $ (recall that $ T^{100L\de}_{c(v),v}$ is the $100L\delta$-neighborhood of the segment $\{ (c(v)+tv,t): t\in [0,1]\} $). Therefore, \begin{equation}\label{lastpa}
	    \im(\phi)\subset \bigcup_{v\in V_\de}T^{100L\delta}_{c(v),v}.
	\end{equation}
% 	It suffices to prove $|\im(\phi)|\gtrsim (\|c\|_{\Lip}+1)^{-(n-2)(n-1)}$.
	
	Choose a maximal $\delta$-separated subset of $ B^{n-1}(0,100L\delta) $, denoted by $ \{ x_1,\cdots, x_M \} $ with $ M\sim L^{n-1} $. We see for any $ v\in V_\de $, 
	$$ T^{100L\delta}_{c(v),v}\subset\bigcup_{1\leq i\leq M} T^{\de}_{c(v)+x_i,v}. $$
	Combined with \eqref{lastpa}, we have $$\im (\phi)\subset \bigcup_{1\leq i\leq M}\bigcup_{v\in V_\de}T^{\de}_{c(v)+x_i,v}. $$
	We also observe that $ \bigcup_{v\in V_\de}T^{\de}_{c(v)+x_i,v}=\bigcup_{v\in V_\de}T^{\de}_{c(v),v}+(x_i,0) $, which implies $$ |\bigcup_{v\in V_\de}T^{\de}_{c(v)+x_i,v}|=|\bigcup_{v\in V_\de}T^{\de}_{c(v),v}|,$$
	for any $x_i$.
	As a result, we have 
	$$ |\im(\phi)|\le M |\bigcup_{v\in V_\de}T^{\de}_{c(v),v}|\lesssim \|c\|_{\Lip}^{n-1}|\bigcup_{v\in V_\de}T^{\de}_{c(v),v}|.$$
	
	It remains to find a lower bound of $|\im(\phi)|$. We note the following inequalities
	\begin{align*}
	    1&\lesssim \int_{0}^{1}\int_{\R^{n-1}} |\wind_{t}(x)| dxdt\\
	    &\lesssim (\int_{0}^{1}\int_{\R^{n-1}} |\wind_{t}(x)|^{\frac{n-1}{n-2}} dxdt)^{\frac{n-2}{n-1}}|\im(\phi)|^{\frac{1}{n-1}}\\
	    &\lesssim (\int_0^1 A(\phi_t|_{S^{n-2}})^{\frac{n-1}{n-2}}dt)^{\frac{n-2}{n-1}}|\im(\phi)|^{\frac{1}{n-1}}\\
	    &\lesssim (\|c\|_{\Lip}+1)^{n-2} |\im(\phi)|^{\frac{1}{n-1}}
	\end{align*} 
	Here, the first inequality is by Proposition \ref{deglowerbnd} , the second inequality is  H\"older's inequality, the third inequality is by the isoperimetric inequality (Lemma \ref{lem2}), and the fourth inequality is by $A(\phi_t|_{S^{n-2}})\lesssim \|\det \sqrt{D\phi_t^*D\phi_t}\|_\infty\lesssim \|D\phi_t\|_\infty^{n-2}\lesssim (\|c\|_{\Lip}+1)^{n-2}$.
	
	We obtain 
	$$ |\bigcup_{v\in V_\de}T^{\de}_{c(v),v}|\gtrsim \|c\|_{\Lip}^{-(n-1)}|\im(\phi)|\gtrsim (\|c\|_{\Lip}+1)^{-(n-1)^2}. $$
\end{proof}

Furthermore we may ask the estimate for the tube-Kakeya set under the $\alpha$-H\"older regularity assumption. Motivated by Theorem \ref{holderthm}, we wonder whether the following is true.

\begin{question}\label{lipkakeya}
Does there exists $\alpha<1$ so that
\begin{equation}\label{eqlipkakeya}
    |\bigcup_{v\in V_\de}T^{\de}_{c(v),v}|\gtrsim_{\e,\|c\|_{C^\alpha}}\de^\e, 
\end{equation}
for any $\e>0$ ?
\end{question}

\subsection{Line-Kakeya set}

In this subsection we would like to use a slightly different notation for the Kakeya map. The direction set is parametrized by $S^{n-1}$, as opposed to Definition \ref{kakeyamaplocal} where the direction set is parametrized by $B^{n-1}(0,1)$. 

\begin{definition}
    For any function $ c:S^{n-1}\rightarrow \R^n $, we let $ \phi_c$ be the map 
    $$\phi_c:S^{n-1}\times [0,1]\rightarrow \R^{n},\qquad (v,t)\mapsto c(v)+tv.$$
    We call $ \phi_c $ a Kakeya map, and call $K_c:=\im(\phi_c)=\bigcup_{v\in S^{n-1}}c(v)+[0,1]\cdot v$ the associated {{Kakeya set}}. 
\end{definition}

We can also define the line-Kakeya set where line segments are replaced by infinite lines
$$ \tilde{K}_{c}:= \bigcup_{v\in S^{n-1}} c(v)+\R_{\geq 0} \cdot v.$$

\begin{prop}\label{linekakeya}
	If $ c $ is continuous and $\im (c)\subset B(0,R) $, then $ \tilde{K}_{c}\supset\R^n\backslash B(0,R) $.
\end{prop}

\begin{proof}
	We prove it using the degree theory from topology. For any point $ x\notin B(0,R) $, we define a map
	$$ f:S^{n-1}\rightarrow S^{n-1} $$
	$$ f(v)=\frac{x-c(v)}{|x-c(v)|}. $$
	We see that $ f $ is not surjective (actually $\im(f)$ is contained in a half sphere), and hence $ f $ has degree 0, which implies $ f $ has a fixed point (see for example Section 2.2 of \cite{hatcher2002algebraic}). Let $v$ be a fixed point of $f.$ Then
	$$ v=\frac{x-c(v)}{|x-c(v)|} $$
	or equivalently,
	$$ x=c(v)+|c(v)-x|\cdot v\in \tilde{K}_{c}.$$
\end{proof}

We could immediately obtain the following result for a segment-Kakeya set provided that $ c $ has small image.
\begin{prop}
	If $ c $ is continuous and $\mathrm{diam}(\im (c))<\frac{1}{2} $. Then $ K_{c} $ has positive Lebesgue measure.
\end{prop}

\begin{proof}
By the assumption, we have $ c(S^{n-1}) \subset B_{0.9} $, a ball of radius 0.9. By Proposition \ref{linekakeya}, we have $ \tilde{K}_{c}\supset\R^n\setminus B_{0.9}$. Therefore $ K_{c}\supset B_1\setminus B_{0.9}. $   
\end{proof}

We could also prove a local version of the theorem. To be precise, for each subset $ U $ of $ S^{n-1} $ and a continuous $ c $ as above, we define the line-Kakeya set with directions in $ U $ to be $ \tilde{K}_{c}(U)= \bigcup_{v\in U} c(v)+\R_{\geq 0} \cdot v.$
\begin{prop}
	If $ B_r $ is a small closed ball of radius $ r $ in $ S^{n-1} $, then $  \tilde{K}_{c}(B_r) $ contains an infinite cone with the cone angle $ \gtrsim r.$
\end{prop}

\begin{proof}
    Suppose $ c(B_r) $ lies in $ B_R $($ \subset \R^n $), a ball of radius $ R $. We can find another continuous map $ \Tilde{c} $ on the whole sphere $ S^{n-1} $, such that $ \Tilde{c}(S^{n-1})\subset B_R$ and $\Tilde c|_{B_r}=c|_{B_r}$.
    Without loss of generality, we assume the center of $ B_R $($ \subset \R^n $) is 0 and the center of $ B_r $($ \subset S^{n-1} $) is the north pole $ (0,\cdots,0,1) $ of $ S^{n-1}. $ 
    
    Consider the cone $ C=\{(\bar{x},x_n)\in\R^n:x_n-\frac{R}{r}>\frac{|\bar{x}|}{r} \} $ which is  the shaded region in Figure \ref{conefig}. By Proposition \ref{linekakeya}, $ C\subset \tilde{K}_{\Tilde{c}} $. Also note that for any $ x\in C $ and $ y\in B_R $, we have $ \frac{x-y}{|x-y|}\in B_r\subset S^{n-1} $, so actually we have $ C\subset \tilde{K}_{\Tilde{c}}(B_r)=\tilde{K}_{c}(B_r).$
\end{proof}

\begin{figure}

\centering

\begin{tikzpicture}
\draw  (-1,-2) -- (1,2);
\draw (-1,2) -- (1,-2);
\draw (0,-1.25) circle (0.559); 
\node at (0,-1.25) {$B_R$};
\filldraw [color=gray!20] (0,0.04) -- (0.96,2) -- (-0.96,2) -- cycle;
\draw [ ->] (0.2,0.4) arc 
(63.434:116.566:0.447) node[above]{$r$} ;
% (63.434:116.566:0,447);
\end{tikzpicture}

\caption{}
\label{conefig}
% label is given after caption
\end{figure}

\vspace{5mm}
\noindent
{\bf Acknowledgements.} We would like to thank Larry Guth for suggesting this problem and many helpful discussions, and Nets Katz for helpful comments.

\bibliographystyle{abbrv}
\bibliography{refs}

\vspace{5mm}

\end{document}